\documentclass{article}
\usepackage[T1]{fontenc}
\usepackage[english]{babel}
\usepackage{textcomp,lmodern}
\usepackage{mathtools,amssymb,amsthm}
\newtheorem{lemma}{Lemma}[section]
\newtheorem{corollary}[lemma]{Corollary}
\newtheorem{proposition}[lemma]{Proposition}
\newtheorem{theorem}[lemma]{Theorem}
\newtheorem*{theoremmain}{Theorem~\ref{Thm:MainThm}}\theoremstyle{definition}
\newtheorem{definition}[lemma]{Definition}
\newtheorem{remark}[lemma]{Remark}
\newtheorem{example}[lemma]{Example}
\DeclareMathOperator\Aut{Aut}
\DeclareMathOperator\diag{diag}
\DeclareMathOperator\id{id}
\DeclareMathOperator\Rist{Rist}
\DeclareMathOperator\Stab{Stab}
\newcommand*\abs[1]{\lvert#1\rvert}
\newcommand*\gen[1]{\langle#1\rangle}
\newcommand*{\Grig}{{\mathfrak G}}
\newcommand*\level[1]{\mathcal L_{#1}}
\newcommand*\setst[2]{\{#1\,|\,#2\}}
\newcommand*\subdirect{\leq_{\mathrm{s}}}
\newcommand*{\restr}[2]{\pi_{{#2}}(#1)}
\usepackage[colorlinks,breaklinks,bookmarks,plainpages=false,unicode=true]{hyperref}  
\title{Finitely generated subgroups of branch groups and subdirect products of just infinite groups}
\author{Rostistlav Grigorchuk, Paul-Henry Leemann, Tatiana Nagnibeda}
\date{\today}
\hypersetup{pdftitle={Finitely generated subgroups of branch groups and subdirect products of just infinite groups}, pdfauthor={Rostistlav Grigorchuk, Paul-Henry Leemann, Tatiana Nagnibeda}, pdfsubject={branch groups, finitely generated subgroups, subdirect product, Grigorchuk group, Gupta-Sidki groups, GGS groups, subgroup induction theorem, just infinite group, block structure}}
\begin{document}
\maketitle
\begin{flushright}
\textit{In memory of Sergei Ivanovich Adian}
\end{flushright}
\begin{abstract}
The aim of this paper is to describe the structure of the finitely generated subgroups of a family of branch groups, which includes the first Grigorchuk group and the Gupta-Sidki $3$-group.
We then use this to show that all groups in the above family are subgroup separable (LERF).

These results are obtained as a corollary of a more general structural statement on subdirect products of just infinite groups.
\end{abstract}
%
%
%
%
%
%
%
%
%
%
\section{Introduction}
A group is \emph{branch} if it acts faithfully on a spherically homogeneous rooted tree and has the lattice of subnormal subgroups similar to the structure of the tree~\cite{MR1754662,MR1820639,MR1765119}. A group $G$ is \emph{self-similar} if it has a faithful action on a $d$-regular rooted tree, $d\geq 2$, such that any section of any element $g\in G$ is again an element of the group modulo the canonical identification of the subtree and the original tree.
Just infinite branch groups constitute one of three classes of \emph{just infinite} groups (infinite groups whose proper quotients are all finite)~\cite{MR1765119}. Self-similar groups appear naturally in holomorphic dynamics~\cite{MR2162164}. 
Both classes are also important in many other areas of mathematics.
Although quite different, these two classes of groups have large intersection, and many self-similar groups are also branch. In the class of finitely generated branch self-similar groups, there are torsion groups and torsion free groups; groups of intermediate growth and groups of exponential growth; non-elementary amenable and nonamenable groups. Branch self-similar groups have very interesting subgroup structure.
Precise definitions, more details, and relevant references can be found in~\cite{MR1841755,MR2162164}.

Among the most important examples of branch self-similar groups is the $3$-generated $2$-group $\Grig$ of intermediate growth~\cite{MR712546} which is defined by its action on the rooted binary tree $T$.
See~\cite{MR1786869} for an introduction to this group and~\cite{MR2195454} for detailed information and a list of open problems about~it.
Much is known about subgroups of $\Grig$, in particular about the stabilizers of vertices of the tree $T$ and of points on the boundary of $T$; the rigid stabilizers; the centralizers; certain subgroups of small index~\cite{MR1872621,LeemannThese}, maximal subgroups~\cite{MR1841763} as well as \emph{weakly maximal subgroups} (subgroups of infinite index that are maximal for this property)~\cite{L2019}.

In~\cite{GN-Oberwolfach} the first and the third named authors announced a structural result for finitely generated subgroups of $\Grig$ in terms of a new notion of block subgroups, see Definition~\ref{Def:BlockStructureBranch} and Example~\ref{example:block} illustrated by Figure~\ref{Figure:Block}.
\begin{figure}[htbp]
\centering
\includegraphics{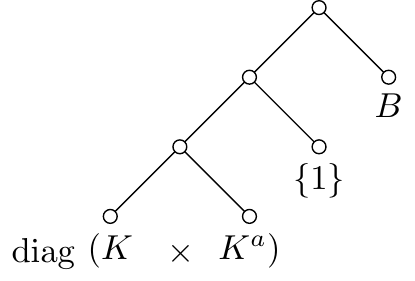}
\caption{A block subgroup of the first Grigorchuk group.}
\label{Figure:Block}
\end{figure}

The main purpose of this article is to prove a weak version of this result in a more general context, which applies in particular to the Gupta-Sidki $3$-group~$G_3$, see Theorem~\ref{Thm:MainThm}.
We will also derive some important consequences of~it, notably that the groups in question are subgroup separable, see Theorem~\ref{Thm:LERF}.
The desired result for branch groups will follow from a more general result of independent interest about subdirect products of just infinite groups, see Theorem~\ref{Thm:SubdirectBloc}.

The article is organized as follows.
The next section contains the definitions and the principal results. It is subdivided in two subsections, the first one deals with branch groups while the second one is about general subdirect products.
Section~\ref{Section:Subdirect} contains the proofs of the general statements on subdirect products and notably of Theorem ~\ref{Thm:SubdirectBloc}.
Finally, Section~\ref{Section:ProofsBranch} is devoted to the proof of the results concerning branch groups, in particular of Theorem~\ref{Thm:MainThm}.

\paragraph{Acknowledgements}
We are grateful to Dominik Francoeur for numerous valuable remarks on the first version of this paper.
The authors gratefully acknowledge support of the Swiss National Science Foundation.
The second author performed this work within the framework of the LABEX MILYON (ANR-10-LABX-0070) of Universit\'e de Lyon, within the program ``Investissements d'Avenir'' (ANR-11-IDEX-0007) operated by the French National Research Agency (ANR).
The first and the third authors were partly supported by the grant of the Government of the Russian Federation No 14.W03.31.0030.
%
%
%
%
%
%
%
%
%
%
\section{Definitions and results}
%
%
%
%
%
\subsection{Branch groups}\label{subsection:BranchGroups}
The main goal of this research is to understand the structure of subgroups of the first Grigorchuk group $\Grig$ and of the Gupta-Sidki $3$-group $G_3$, as well as of some other branch groups, that are closed in profinite topology (observe that branch groups are residually finite). One class of such subgroups consists of finitely generated subgroups, as proven in~\cite{MR2009443} and ~\cite{MR3513107}. It is shown there that every infinite finitely generated subgroup of $\Grig$ (respectively $G_3$) is commensurable\footnote{Recall that two abstract groups $G_1$ and $G_2$ are \emph{(abstractly) commensurable} if there exists $H_i$ of finite index in $G_i$ such that $H_1$ and $H_2$ are isomorphic. In particular, if $H$ is a finite index subgroup of $G$, then $H$ is commensurable to $G$.} with $\Grig$ (respectively with $G_3$ or $G_3\times G_3$). This unusual property relies on the fundamental result of Pervova~\cite{MR1841763} that every maximal subgroup of $\Grig$ (respectively of the Gupta-Sidki $p$-group $G_p$) has finite, hence $= 2$ (respectively $=p$), index.
For strongly self-replicating (see Definition~\ref{Def:SelfRepl}) just infinite groups with the congruence subgroup property, the property to have all maximal subgroups of finite index is preserved when passing to commensurable groups~\cite[Lemma 4]{MR2009443}, and thus weakly maximal subgroups in $\Grig$ (respectively $G_p$) are closed in the profinite topology.
For a branch group $G$, the stabilizers of points in the boundary of the tree are examples of weakly maximal subgroups~\cite{MR1841750}, but there are much more, see~\cite{MR3478865}.
See~\cite{L2019} for a description of all weakly maximal subgroups of $\Grig$ and~$G_p$.

Let $T$ be \emph{$d$-regular rooted tree}.
This is a connected graph without cycles, with a distinguished vertex, the root, of degree $d$ and such that every other vertex has degree $d+1$.
The vertices of $T$ are in natural bijection with the free monoid $\{0,\dots,d-1\}^*$, which is the set of finite words in the alphabet $\{0,\dots,d-1\}$.
Under this identification, the root corresponds to the empty word $\emptyset$.
There is a natural partial order on the vertices of $T$ defined by $v\leq w$ if $v$ is a prefix of $w$, equivalently if the unique path without backtracking from the root to $w$ hits $v$ before $w$.
The \emph{$n$\textsuperscript{th} level}, $\level{n}$, of $T$ is the set of all vertices at distance $n$ from the root, or equivalently, the set of all words of length $n$.
For a vertex $v$ of $T$, we denote by $T_v$ the subtree of $T$ consisting of all vertices $w\geq v$, naturally rooted at $v$.
By $\Aut(T)$ we denote the automorphism group of $T$, that is, the set of all graph isomorphisms from $T$ to itself that preserve the root.
Equivalently $\Aut(T)$ can be seen as the set of bijections of the free monoid $\{0,\dots,d-1\}^*$ that preserve length and prefixes.\footnote{A bijection $\varphi$ of $\{0,\dots,d-1\}^*$ preserves prefixes if for all words $u,v$ there exists $v'$ such that $\varphi(uv)=\varphi(u)v'$.}

Let $G$ be a subgroup of $\Aut(T)$. We denote by $\Stab_G(v)$ the \emph{stabilizer of the vertex $v$} and $\Stab_G(\level{n})$ the \emph{$n$\textsuperscript{th} level stabilizer}, that is, the pointwise stabilizer $\Stab_G(\level{n})\coloneqq\bigcap_{v\in\level{n}}\Stab_G(v)$.
More generally, if $X$ is a subset of vertices of $T$, the group $\Stab_G(X)$ is its pointwise stabilizer.
Other important subgroups are \emph{the rigid stabilizer of a vertex} $\Rist_G(v)\coloneqq\bigcap_{w\notin T_v}\Stab_G(w)$ which consists of elements acting trivially outside $T_v$ and \emph{the rigid stabilizer of a level} $\Rist_G(n)\coloneqq\prod_{v\in\level{n}}\Rist_G(v)=\langle\Rist_G(v)\,|\,v\in\level{n}\rangle$.

For $v$ a vertex of $T$ we have a natural homomorphism
\[
	\varphi_v\colon\Stab_{\Aut(T)}(v)\to\Aut(T_v)
\]
where $\varphi_v(G)=g_{|_{T_v}}$ is simply the restriction of $g$, i.e the action of $g$ on $T_v$.
The element $\varphi_v(G)$, is called the \emph{section of $g$ at $v$}.
For $G$ a subgroup of $\Aut(T)$, we will sometimes speak of the \emph{section of $G$ at $v$} and write $\varphi_v(G)$ as a shorthand for $\varphi_v\bigl(\Stab_G(v)\bigr)$.
It is trivial that the restriction of $\varphi_v$ to $\Rist_G(v)$ is an isomorphism onto its image.
\begin{definition}\label{Def:SelfRepl}
A subgroup $G$ of $\Aut(T)$ is \emph{self-similar} if for every vertex $v$ the section $\varphi_v(G)$ is, under the natural identification of $T_v$ to $T$,  a subgroup of $G$.
It is \emph{self-replicating} (also called \emph{fractal}) if for every vertex $v$ the section $\varphi_v(G)$ is, under the natural identification of $T_v$ to $T$, equal to $G$.
Following~\cite{MR4082048} we will say that $G$ is \emph{strongly self-replicating} if $\varphi_v\bigl(\Stab_G(\level{n})\bigr)=G$ for every vertex $v$ of level $n$.
\end{definition}
\begin{definition}
A subgroup $G$ of $\Aut(T)$ is \emph{weakly branch} if it acts transitively on $\level{n}$ for all $n$ and all the $\Rist_G(v)$ are infinite (equivalently, they are all non-trivial).
The group $G$ is \emph{branch} if it acts transitively on $\level{n}$ for all $n$ and all the $\Rist_G(n)$ have finite index in $G$.
Finally, $G$ is \emph{regular branch over a subgroup $K$} if it acts transitively on $\level{n}$ for all $n$ and is self-replicating, $K$ is finite index and for every vertex $v$ of the first level, $\varphi_v\bigl(\Stab_K(\level{1})\bigr)$ contains $K$ as a finite index subgroup.
\end{definition}
Regularly branch groups are branch, and branch groups are weakly branch.

A branch group $G$ has the \emph{congruence subgroup property} if for every finite index subgroup  $H$ there exists $n$ such that  $H$ contains the level stabilizer $\Stab_G(\level{n})$.
This is equivalent to the fact that the profinite topology on $G$ coincides with the restriction to $G$ of the natural topology of $\Aut(T)$.

The \emph{first Grigorchuk group} $\Grig$ acting on the $2$-regular rooted tree is probably the best-known and most studied branch group. This was the first example of a group of intermediate growth~\cite{MR764305}.
Among other properties, the group $\Grig$ is self-replicating, regularly branch (over an index $16$ subgroup denoted $K$), just infinite and has the congruence subgroup property.
See~\cite{MR1786869,MR2195454} for a formal definition, references and more details.

Other well-studied examples of branch groups are the \emph{Gupta-Sidki  $p$-groups} $G_p$, $p\geq 3$ prime, acting on the $p$-regular rooted tree~\cite{MR696534}.
These groups are self-replicating, regular branch, just infinite~\cite{MR759409} and have the congruence subgroup property~\cite{MR1899368}.
On the other hand, Gupta-Sidki groups have infinite width of associated Lie algebras, in contrast with $\Grig$, and it is not know if they have intermediate growth.
See~\cite{MR696534,MR759409,MR3513107} for a formal definition, references and more details.

Let us now introduce some further, less standard terminology that we are going to work with in this paper.

We shall say that two vertices $u$ and $v$ of $T$ are \emph{orthogonal} if the subtrees $T_u$ and $T_v$ do not intersect, that is, if both $v\nleq w$ and $w\nleq v$. A subset $U$ of vertices is called \emph{orthogonal} if it consists of pairwise orthogonal vertices. It is called a \emph{transversal} if every infinite geodesic ray from the root of the tree intersects $U$ in exactly one point. It is clear that a transversal is a finite set. Two subsets $U$ and $V$ of vertices are \emph{orthogonal} if every vertex of one set is orthogonal to every vertex of the other set.
\begin{definition}\label{Def:Diagonal}
Let  $U = (u_1,\dots,u_k)$ be an ordered orthogonal set.
Let $G\leq\Aut(T)$ and let $L$ be an abstract group. Suppose that there exists a family $(L_j)_{j=1}^k$ of finite index subgroups of $\varphi_{u_j}\bigl(\Rist_G(u_j)\bigr)$ that are all isomorphic to~$L$.
Let $\Psi=(\psi_1,\dots,\psi_k)$ be a $k$-uple of isomorphisms $\psi_j\colon L\to L_{j}$.
Then the quadruple $(U,L,(L_j)_{j=1}^k,\Psi)$ determines a \emph{diagonal subgroup} of~$G$
\[
	D\coloneqq\biggl\{g\in\prod_{j=1}^k\Rist_G(u_j)\bigg|\exists l\in L,\forall j:\varphi_{u_j}(g)=\psi_j(l)\biggr\}
\]
abstractly isomorphic to $L$. We say that $U$ is the \emph{support} of $D$.
\end{definition}
Observe that for the degenerate case where $U=\{u\}$ consists of only one vertex, the diagonal subgroups of $G$ supported on $\{u\}$ are exactly finite index subgroups of $\Rist_G(u)$.
\begin{example}\label{Ex:Grig}
The first Grigorchuk group $\Grig=\gen{a,b,c,d}$ is regular branch over $K=\gen{abab}^\Grig=\gen{abab,badabada,abadabad}$ which is a normal subgroup of index~$16$.
The diagonal subgroup $D$ depicted in Figure~\ref{Figure:Diag} is defined by $L\coloneqq K$ (seen as an abstract group), $U=\{000,01,10\}$, $3$ copies of $K$ living respectively in $\Aut(T_{000})$, $\Aut(T_{01})$ and $\Aut(T_{10})$, and the isomorphisms $\psi_i$, $i=1,2,3$, given  by the conjugation by $a$, $b$ and $c$.
That is, $\psi_1\colon K\to K\leq \Aut(T_{000})$ is defined by $\psi_1(g)=g^a$, hence the notation $K^a$ in Figure~\ref{Figure:Diag}; and similarly for $\psi_2$ and $\psi_3$.
Observe that in this example the subgroup $D$ is ``purely'' diagonal in the sense that each factor $K$ coincides with $\varphi_{u_j}\bigl(\Rist_G(u_j)\bigr)$ instead of only being a finite index subgroup of it.
\begin{figure}[htbp]
\centering
\includegraphics{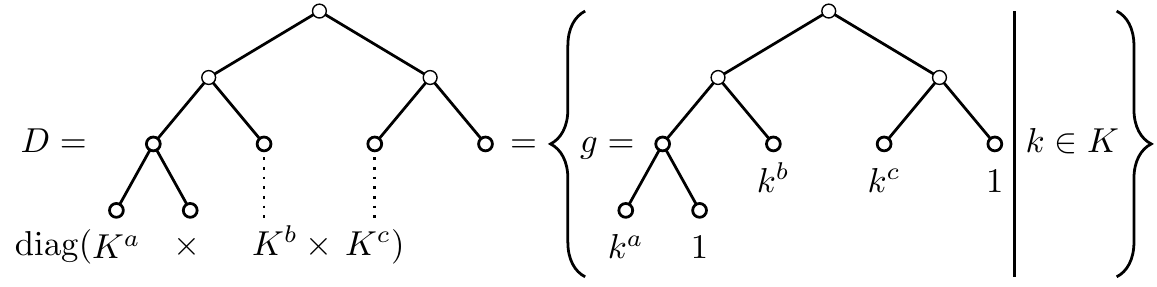}
\caption{A diagonal subgroup of $\Grig$.}
\label{Figure:Diag}
\end{figure}
\end{example}
\begin{definition}\label{Def:BlockStructureBranch}
Let $G\leq \Aut(T)$.
A \emph{block subgroup} of $G$ is a finite product $A=\prod_{i=1}^nD_i$ of diagonal subgroups such that the supports of the $D_i$s are pairwise orthogonal.
\end{definition}
Observe that since the supports of the $D_i$s are pairwise orthogonal, we have $D_i\cap D_j=\{1\}$ if $i\neq j$ and $\prod_{i=1}^nD_i=\gen{D_1,\dots,D_n}\leq G$.
\begin{example}
\label{example:block}
A picture describing a specific block subgroup of $\Grig$ is given in Figure~\ref{Figure:Block}.
Recall that $\Grig=\gen{a,b,c,d}$ is regular branch over the normal subgroup $K=\gen{abab}^\Grig$ and we have $K<B=\gen{b}^\Grig<\Grig$ with $[B:K]=2$ and $[\Grig:B]=8$.
Moreover, the section $\varphi_v\bigl(\Rist_G(v)\bigr)$ is equal to $B$ if $v\in\level{1}$ and to $K$ otherwise.

Here is a detailed explanation of this example.
We have $U_1=\{1\}$ and $D_1=\Rist_G(1)=\{1\}\times B$.
On the other hand, we have $U_2=\{000,001\}$, $L_{1}=K\leq\Aut(T_{000})$ and $L_{2}=K\leq\Aut(T_{001})$ with the isomorphisms $\psi_1=\id\colon K\to L_1$ and $\psi_2=\cdot^a\colon K\to L_2$ (the conjugation by $a$). This gives us $D_2=\setst{(g,aga,1,1,1,1,1,1)\in \Rist_{G}(000)\times \Rist_{G}(001)}{g\in K}$.
Finally, the block subgroup depicted in Figure~\ref{Figure:Block} is the product of $D_1$ and $D_2$ (these two subgroups have trivial intersection).
\end{example}
It follows from the definition that if $G$ is a finitely generated branch group and $A$ is a block subgroup of $G$, then $A$ is finitely generated and so is $H$ for every subgroup $A\leq H\leq G$ with $[H:A]<\infty$.
We will show that the converse also holds under some conditions.
In order to do that, we need two more definitions.
\begin{definition}\label{Definition:InductiveClass} Let $G\leq \Aut(T)$ be a self-similar group.
A family $\mathcal{X}$ of subgroups of $G$ is said to be \emph{inductive} if
\begin{enumerate}\renewcommand{\theenumi}{\Roman{enumi}}
\item\label{Item:DefSubgroupInduction1}
Both $\{1\}$ and $G$ belong to $\mathcal{X}$,
\item\label{Item:DefSubgroupInduction2}
If $H\leq L$ are two subgroups of $G$ with $[L:H]$ finite, then $L$ is in $\mathcal X$ if and only if $H$ is in $\mathcal X$,
\item\label{Item:DefSubgroupInduction3}
If $H$ is a finitely generated subgroup of $\Stab_G(1)$ and all first level sections of $H$ are in $\mathcal{X}$, then $H\in \mathcal{X}$.
\end{enumerate}
\end{definition}
\begin{definition}\label{Definition:SubgroupInduction}
A self-similar group $G$ has the \emph{subgroup induction property} if for any inductive class of subgroups $\mathcal X$, each finitely generated subgroup of $G$ is contained in $\mathcal X$.
\end{definition}
We are now able to state our main result.
\begin{theorem}\label{Thm:MainThm}
Let $G$ be a finitely generated self-similar branch group. Then the following are equivalent.
\begin{enumerate}
\item 
The group $G$ has the subgroup induction property,
\item
A subgroup $H$ of $G$ is finitely generated if and only if it contains a block subgroup $A$ with $[H:A]<\infty$,
\item
A subgroup $H$ of $G$ is finitely generated if and only if there exists $n$ such that for every $v\in \level{n}$ the section $\varphi_v\bigl(\Stab_H(\level{n})\bigr)$ is either trivial or has finite index in $\varphi_v\bigl(\Stab_G(v)\bigr)$.
\end{enumerate}
\end{theorem}
The proof of Theorem~\ref{Thm:MainThm} will be given in the end of Section~\ref{Section:ProofsBranch}.

The first Grigorchuk group~\cite[Theorem 3]{MR2009443}, as well as torsion GGS groups~\cite{FL2019} are known to possess the  subgroup induction property.
Hence it follows from Theorem~\ref{Thm:MainThm}: 
\begin{corollary}\label{cor:MainCor}
Let $\mathcal G$ be either the first Grigorchuk group, or a torsion GGS group.
Then a subgroup $H$ of $\mathcal G$ is finitely generated if and only if it contains a block subgroup $A$ with $[H:A]<\infty$.
\end{corollary}
The property of a group to have all finitely generated subgroups closed in profinite topology is quite rare.
Such groups are called \emph{subgroup separable} (or \emph{LERF}, which stands for locally extensively residually finite).
It holds in particular for finite groups, finitely generated abelian groups, finitely generated free groups~\cite{MR32642}, surface groups~\cite{MR494062}, amalgamated product of two free groups over a cyclic subgroup~\cite{MR767102} and more generally, limit groups~\cite{MR2399104}.
In~\cite{MR1190361} it is proven that a subset of a free group which is a product of finitely many finitely generated subgroups is closed in profinite topology. This remarkable property is known (in finitely generated case) only for free groups.
We conjecture that every subset of $\Grig$ (respectively of $G_3$) which is a product of finitely many finitely generated groups is closed in the profinite topology. Our result may be considered as positive evidence towards this conjecture.
Indeed,  if $G\leq Aut(T)$ is a branch group with the congruence subgroup property, then every block subgroup of $G$ is closed in the profinite topology~\cite{L2019}.
Together with Theorem~\ref{Thm:MainThm} this implies the following.
\begin{theorem}\label{Thm:LERF}
Let $G$ be a finitely generated self-similar branch group with the congruence subgroup property.
If $G$ has the subgroup induction property, it is subgroup separable.
\end{theorem}
As a corollary, we obtain another proof of the fact that both $\Grig$~\cite{MR2009443} and~$G_3$~\cite{MR3513107} are subgroup separable.

The first Grigorchuk group~\cite{MR2009443} and the Gupta–Sidki $3$-group~\cite{MR3513107} were the only groups known to have the subgroup induction property, until the very recent work~\cite{FL2019} of the second named author together with D.~Francoeur generalizing the results of~\cite{MR3513107} to all Gupta-Sidki $p$-groups, ($p\geq 3$ prime) and more generally to torsion GGS groups acting on the $p$-regular rooted tree.
%
%
%
%
%
%
%
%
%
%
\subsection{Subdirect products}
A subgroup $H\leq\prod_{i\in I}G_i$ that projects onto each factor is called a \emph{subdirect product}, written $H\subdirect \prod_{i\in I}G_i$.

Subdirect products have been used for example to represent many small perfect groups~\cite{MR1025760}.
On the other hand, understanding subdirect products is the first step in the comprehension of general subgroups of direct products and is therefore of interest in itself.

In this context, diagonal subgroups play an important role.
We need to introduce the notion of a diagonal subgroup in the abstract setting of a direct product which will be coherent with the notion of diagonal subgroups introduced in Definition~\ref{Def:Diagonal}.
\begin{definition}
Let $I$ be a set and let $(G_i)_{i\in I}$ be a collection of groups that are all isomorphic to a common group $G$.
We say that a subgroup $H$ of $\prod_{i\in I}G_i$ is \emph{diagonal} if there exists isomorphisms $\psi_i\colon G\to G_i$ such that
\[
H=\diag\Bigl(\prod_{i\in I}\psi_i(G)\Bigr)\coloneqq\setst{\bigl(\psi_i(g)\bigr)_{i\in I}}{g\in G}.
\]
It directly follows from the definition that diagonal subgroups of $\prod_{i\in I}G_i$ are examples of subdirect product.
\end{definition}

We will show in Section~\ref{Section:Subdirect} that for just infinite groups diagonal subgroups together with $G\times G$ itself are basically the only examples of subdirect products of $G\times G$.
\begin{lemma}\label{Lemma:subdirectProductJustInf}
Let $G_1$ and $G_2$ be two just infinite groups and $H\subdirect G_1\times G_2$.
Then either there exists $D_i\leq G_i$ of finite index such that $H$ contains $D_1\times D_2$, or $G_1$ and $G_2$ are isomorphic and $H=\diag\bigl(G_1\times\psi(G_1)\bigr)$ for some isomorphism $\psi\colon G_1\to G_2$.
\end{lemma}
For products of just infinite groups with more than two factors, full products and diagonal subgroups are the building blocks of subdirect products.
We need one last definition before stating our result.
\begin{definition}\label{Def:BlockStructure}
Let $(G_i)_{i\in I}$, be a family of groups indexed by a set $I$.
A subgroup $H$ of $\prod_{i\in I}G_i$ is \emph{virtually diagonal by blocks} if there exists a set $\Delta$, 
abstract groups $(G_\alpha)_{\alpha\in \Delta}$, finite index subgroups $(L_\alpha\subset G_\alpha)_{\alpha\in \Delta}$ and subgroups $(D_\alpha)_{\alpha\in \Delta}$ of $\prod_{i\in I_\alpha} G_i$ such that I is partitioned into $I=\bigsqcup_{\alpha\in\Delta} I_\alpha$ with
\begin{enumerate}
\item For every $i$ in $I_\alpha$, the group $G_i$ is isomorphic to $G_\alpha$,
\item The subgroup $D_\alpha$ is a diagonal subgroup of $L_\alpha^{\abs{I_\alpha}}$,
\item $H$ contains $\prod_{\alpha\in\Delta} D_\alpha$ as a finite index subgroup.
\end{enumerate} 
\end{definition}
Observe that it follows from the definition that if $I_\alpha$ contains only one element, then $D_\alpha$ is  a finite index subgroup of $G_\alpha$.

We are now able to state our main result about subdirect products.
\begin{theorem}\label{Thm:SubdirectBloc}
Let $G_i$, $1\leq i\leq n$, be just infinite groups.
Suppose that at most two of them are virtually abelian.
Then all subdirect products of $\prod_{i=1}^nG_i$ are virtually diagonal by blocks.
\end{theorem}
The restriction on the number of virtually abelian factors is necessary, as demonstrated by Example~\ref{Ex:counterexample}.
On the other hand, it is not too restrictive in the sense that nearly all just infinite groups, and in particular all torsion just infinite groups, are not virtually abelian.
Indeed, McCarthy showed in~\cite{MR0237637} that if $G$ is a just infinite group, then it has a maximal abelian normal subgroup (either trivial or of finite index) which is equal to its Fitting subgroup\footnote{The Fitting subgroup $\Psi(G)$ is the subgroup generated by the nilpotent normal subgroups of $G$.
If any ascending chain of normal subgroups of $G$ stabilizes after a finite number of steps, then $\Psi(G)$ is the unique maximal normal nilpotent subgroup of $G$.
Indeed, in this case there exists $K_1,\dots,K_n$ normal nilpotent subgroups of $G$ such that $\Psi(G)=\gen{K_1,\dots,K_n}$, which is a normal nilpotent subgroup of $G$ by Fitting's theorem~\cite{Fitting1938}.}
 $\Psi(G)$ and is isomorphic to $\mathbf Z^n$ for some integer $n$.
Such a group $G$ is not virtually abelian if and only if $\Psi(G)=\{1\}$ which correspond to the case $n=0$.
Moreover, for every $n>0$, there exists only finitely many non-isomorphic just infinite groups with $\Psi(G)=\mathbf Z^n$~\cite[Proposition 9]{MR0237637}.
Finally, (weakly) branch groups are never virtually abelian~\cite[Proposition 10.4]{MR2035113}.

If in Theorem~\ref{Thm:SubdirectBloc} the $G_i$s are pairwise non-isomorphic, all the $I_\alpha$s have cardinality $1$ and there exist finite index subgroups $D_i\leq G_i$ such that $H$ contains $\prod_{i=1}^nD_i$.
This directly implies the following rigidity result on subdirect products of non-isomorphic groups.
\begin{corollary}
Let $G_i$, $1\leq i\leq n$, be pairwise non-isomorphic just infinite groups.
Suppose that at most two of them are virtually abelian.
Then every subdirect product of $\prod_{i=1}^nG_i$ is a finite index subgroup.
\end{corollary}
%
%
%
%
%
%
%
%
\section{Block structure in products of just infinite groups}\label{Section:Subdirect}
One important tool for the study of subdirect products is the notion of fibre product.
Recall that if $\psi_1\colon G_1\twoheadrightarrow Q$ and $\psi_2\colon G_2\twoheadrightarrow Q$ are groups epimorphisms, their fibre product\footnote{This is the categorial fibre product, also called \emph{pullback}, of $\psi_1$ and $\psi_2$ in the category of groups.} is the subgroup $P\coloneqq\setst{(g_1,g_2)\in G_1\times G_2}{\psi_1(g_1)=\psi_2(g_2)}$ of $G_1\times G_2$.
Every fibre product is naturally isomorphic to a subdirect product.
The converse also holds.
More precisely, the Goursat Lemma~\cite{MR1508819} applied to subdirect products gives us
\begin{lemma}\label{Lemma:SubdirectFiber}
If $H\subdirect G_1\times G_2$ is a subdirect product, then there exists a quotient $Q$ of $H$ and epimorphisms $\psi_1\colon G_1\twoheadrightarrow Q$ and $\psi_2\colon G_2\twoheadrightarrow Q$ such that $H$ is the fiber product of $\psi_1$ and $\psi_2$.
\end{lemma}
Using this result, we are able to prove Lemma~\ref{Lemma:subdirectProductJustInf}.
\begin{proof}[Proof of Lemma~\ref{Lemma:subdirectProductJustInf}.]
By Lemma~\ref{Lemma:SubdirectFiber}, $H$ is the fibre product of $\psi_1\colon G_1\twoheadrightarrow Q$ and $\psi_2\colon G_2\twoheadrightarrow Q$ where $Q$ is a quotient of $G$.
It follows directly from the definition of a fibre product that $H$ contains the subgroup $\ker(\psi_1)\times\ker(\psi_2)$.
If $Q$ is finite, both $\ker(\psi_1)$ and $\ker(\psi_2)$ have finite index in $G_1$ and $G_2$, respectively, and the conclusion of the lemma holds.
If $Q$ is not finite, by just infinity, both $\ker(\psi_1)$ and $\ker(\psi_2)$ are trivial, and $Q$ is isomorphic to both $G_1$ and $G_2$.
But in this case, $\psi_1$ and $\psi_2$ are isomorphisms and we have
\begin{align*}
H&=\setst{(g,h)\in G_1\times G_2}{\psi_1(g)=\psi_2(h)}\\
&=\setst{(g,h)\in G_1\times G_2}{h=\psi_2^{-1}\circ\psi_1(g)}\\
&=\diag\bigl(G_1\times\psi_2^{-1}\circ\psi_1(G_1)\bigr).
\end{align*}
\end{proof}
\begin{remark}\label{Example:SubdirecNotDiag}
The conclusion of Lemma~\ref{Lemma:subdirectProductJustInf} is optimal in the sense that there exists a subdirect product $P\subdirect G\times G$ such that $P$ is neither diagonal nor equal to $G\times G$.
Indeed, let $G$ be any group that has two distinct subgroups of index~$2$, one of them being characteristic.\footnote{For example, one can take $G=\Grig$ the first Grigorchuk group, $H$ the stabilizer of the first level 
and $J$ one of the other $6$ index $2$ subgroups.}
Then there exists a subdirect product $P\subdirect G\times G$ which is neither a diagonal subgroup nor $G\times G$ itself.

More precisely, for any characteristic subgroup $H$ of index two, and any index two subgroup $J\neq H$, the fibre product $P$ of $G\to G/H$ and $G\to G/J$ is neither a diagonal subgroup nor $G\times G$ itself.
Indeed, we have
\[
	P\coloneqq \setst{(f,g)\in G\times G}{f \in H\iff g\in J}.
\]
Hence $P$ contains $H\times J$ which is an index $4$ subgroup of $G\times G$.
On the other hand, by definition, $P$ is not equal to $G\times G$ (for example, it does not contain $(f,1)$ for any $f$ outside $H$).
Since both projections of $P\subdirect G\times G$ are onto, $P$ strictly contains $H\times J$ and is of index $2$ in $G\times G$.
We claim that $P$ does not contain any diagonal subgroup $\diag\bigl(G\times\psi(G)\bigr)$.
Indeed, suppose that it does.
On one hand, $P$ contains $\diag\bigl(H\times \psi(H)\bigr)$.
On the other hand, $P$ contains $H\times \{1\}$ and therefore it also contains $\{1\}\times \psi(H)=\{1\}\times H$ since $H$ is characteristic.
But $P$ also contains $\{1\}\times J$.
In particular, $P$ contains both $H\times \{1\}$ and $\{1\}\times G$ and is of index $2$ in $G\times G$.
This implies that $P=H\times G$ which is not a subdirect product of $G\times G$ and we have reached the desired contradiction.
\end{remark}
In order to prove Theorem~\ref{Thm:SubdirectBloc} for subdirect products of $\prod_{i\in I}G_i$, we will introduce in Lemma~\ref{Lemma:SubdirectDependent} an equivalence relation on the indices $i\in I$.
In order to do that, we need some notation.
Recall that $\pi_{i}\colon\prod_{i\in I} G_i\to G_i$ is the canonical projection.
\begin{definition}
Let $H\subdirect \prod_{i\in I} G_i$ be a subdirect product.
For $J\subseteq I$ we define $H_J$ by
\[
	H_J\coloneqq\setst{g\in \prod_{i\in I} G_i}{\restr{g}{i}=1 \textnormal{ if } i\notin J\textnormal{ and } \exists h\in H\textnormal{ s.t. }\forall j\in J:\restr{h}{j}=\restr{g}{j}}.
\]

In particular, $H_i\coloneqq H_{\{i\}}\cong G_i$.

Two indices $i$ and $j$ are said to be \emph{dependent} if $i=j$ or if $i\neq j$, $H_i$ and $H_j$ are isomorphic and $H_{\{i,j\}}=\diag\bigl(H_i\times\psi(H_i)\bigr)$ for some isomorphism $\psi\colon H_i\to H_j$.
They are said to be \emph{independent} if $i\neq j$ and $H_{\{i,j\}}$ contains $D_i\times D_j$ where $D_i$, $D_j$ are finite index subgroups of $H_i$, $H_j$.
\end{definition}
Informally, $H_J$ is the projection of $H$ to $\prod_{i\in J} H_i$, but viewed as a subgroup of $\prod_{i\in I} H_i$.

Note that, in general, being independent is not the negation of being dependent, see Example~\ref{Ex:counterexample}.
However, it follows from Lemma~\ref{Lemma:subdirectProductJustInf} that this is the case when the factors are just infinite:
\begin{corollary}\label{cor:dichotomy}
If  $H\subdirect \prod_{i\in I} G_i$ is a subdirect product with all the $G_i$ just infinite, then for all $i$ and $j$, either they are dependent or they are independent.
\end{corollary}
We will now show that being dependent is an equivalence relation and that, for products of just infinite non virtually abelian groups, if $J\subseteq I$ is a finite subset of indices that are pairwise independent, then all indices of $J$ are ``simultaneously independent''.
\begin{lemma}\label{Lemma:SubdirectDependent}
Let $H\subdirect \prod_{i\in I} G_i$ be a subdirect product.
Being dependent is an equivalence relation.
Moreover, if all vertices of $J\subseteq I$ are pairwise dependent, then $H_{J}$ is diagonal.
\end{lemma}
\begin{proof}
Let $J\subseteq I$ be a subset. Suppose that there exists $j_0\in J$ such that every $j\in J$ is dependent with $j_0$.
Then for all $j\in J$ we have an isomorphism $\psi_j\colon G_{j_0}\to G_j$ such that for every $g$ in $H$ we have $\restr{g}{j}=\psi\bigl(\restr{g}{j_0}\bigr)$.
It directly follows that $H_J=\setst{\bigl(\psi_i(g)\bigr)_{i\in J}}{g\in G_j}=\diag\Bigl(\prod_{i\in J}\psi_i(G_j)\Bigr)$.
We hence have proved the second assertion of the lemma.

On the other hand, by taking $J=\{i,j,k\}$ with $(i,j)$ and $(j,k)$ dependent, the above implies (for $j_0=j$) that $i$ and $k$ are dependent.
That is, being dependent is a transitive relation, which is obviously symmetric and reflexive.
\end{proof}
\begin{lemma}\label{Lemma:SubdirectIndependent}
Let $H\subdirect \prod_{i\in I} G_i$ with all the $G_i$s just infinite and such that at most two of them are virtually abelian.
Let $J\subseteq I$ be a finite subset of pairwise independent indices.
Then there exist finite index subgroups $D_j\leq G_j$, $j\in J$, such that $H_{J}$ contains $\prod_{j\in J}D_j$.
\end{lemma}
\begin{proof}
The proof is done by induction on the cardinality of $J$.
If $\abs J=1$, there is nothing to prove.
If $\abs J=2$, this is Lemma~\ref{Lemma:subdirectProductJustInf}.
Now, let $J\subseteq I$ be of cardinality $d$ with $d\geq 3$.
Up to renaming the indices, we may suppose that $J=\{1,\dots, d\}$.
We can moreover suppose that $G_1$ is not virtually abelian, since $n\geq 3$ and at most two of the $G_i$s are virtually abelian.
We claim that the conclusion of the lemma holds if there exist an index $i$  and a finite index subgroup $D_i$ of $H_i$ such that $H_J$ contains $\{1\}\times\dots \times D_i\times\dots \times \{1\}$.
Indeed, let $\tilde H_{J\setminus\{i\}}$ be the subgroup of $H_{J\setminus\{i\}}$ consisting of projections of elements $g$ of $H$ such that $\restr{g}{i}$ is in $D_i$.
Formally, $\tilde H_{J\setminus\{i\}}\coloneqq H_{J\setminus\{i\}}\cap (D_i\times\prod_{j\in J\setminus\{i\}}H_j)$. 
Since $D_i$ has finite index in $H_i$, the subgroup $\tilde H_{J\setminus\{i\}}$ has finite index in $H_{J\setminus\{i\}}$ and, by induction hypothesis, for every $j\in J\setminus\{i\}$ it contains $\tilde D_j$, a finite index subgroup of $H_j$.
On the other hand, $H_J$ contains $D_i\times \tilde H_{J\setminus\{i\}}$, which finishes the proof of the claim.

The subgroup $H_J$ projects onto $H_{\{1,3\dots, d\}}$ which, by induction hypothesis, contains $D_1\times D_3 \dots \times D_{d}$ for some finite index subgroups of $H_i$, $i\neq 2$.
In particular, for every $1\neq g$ in $D_1$, there exist $x$ in $H_2=G_2$ such that $(g,x,1,\dots,1)$ is in $H_J$.
By doing the same argument for $H_{\{1,2,4,\dots, d\}}$ we obtain a finite index subgroup $A$ of $G_1$ such that for every $g$ and $h$ in $A$, the group $H_J$ contains both $(g,x,1\dots,1)$ and $(h,1,y,1\dots,1)$ for some $x$ and $y$.
By taking the commutator and then conjugating by elements of $H_J$, we have that for every $g$ and $h$ in $A$, the subgroup $\gen{[g,h]}^{G_1}\times\{1\}\dots\times\{1\}$ is contained in $H_J$. 
As soon as $[g,h]$ is not trivial, $\gen{[g,h]}^{G_1}$ is a non-trivial normal, and hence finite index, subgroup of $G_1$ and we are done.
Since $G_1$ is not virtually abelian, it is always possible to find $g$ and $h$ in $A$ that do not commute.
\end{proof}
The following example clarifies the conditions imposed in Lemma~\ref{Lemma:SubdirectIndependent} and Theorem~\ref{Thm:SubdirectBloc}.
\begin{example}\label{Ex:counterexample}
Let $G$ be an abelian group (for example $G=\mathbf Z$) and $H_n\coloneqq\{(g_1,\dots,g_n)\in G^n|\sum_i g_i=0\}$.
Then $H_n$ is a subdirect product of $G^n$ such that all indices are pairwise independent, but not simultaneously independent as soon as $n\geq3$.
Moreover, if $n\geq 3$ then $H_n$ is not virtually diagonal by block.
In particular, $H_n$ satisfies neither the conclusion of Lemma~\ref{Lemma:SubdirectIndependent}, nor the conclusion of Theorem~\ref{Thm:SubdirectBloc}.

On the other hand, $H_n$ can be viewed as a subdirect product of $G_1=G$ and $G_2=G^{n-1}$. With this identification and for $n\geq 3$, the indices $1$ and $2$ are neither independent nor dependent.
\end{example}

We now have all the ingredients necessary to prove Theorem~\ref{Thm:SubdirectBloc}.
\begin{proof}[Proof of Theorem~\ref{Thm:SubdirectBloc}.]
Let $G_i$, $1\leq i\leq n$, be just infinite groups such that at most two of them are virtually abelian.
Let $H\subdirect\prod_{i=1}^n G_i$.
By Lemma~\ref{Lemma:SubdirectDependent}, there exists an integer $d$ and a partition $\{1,\dots,n\}=I_1\sqcup\ldots\sqcup I_d$ such that the $I_\alpha$s are equivalence classes of dependent indices.
In particular, $H$ is a subdirect product of $\prod_{\alpha=1}^dH_{I_\alpha}$ and $H_{I_\alpha}\subdirect\prod_{i\in I_\alpha}G_i$ is a diagonal subgroup.
Let $i_\alpha\coloneqq\min\{i\in I_\alpha\}$ and $G_\alpha\coloneqq G_{i_\alpha}$.
Then all the $G_i$s, $i\in I_\alpha$ are isomorphic to $G_\alpha$ and so is $H_{I_\alpha}$.

On the other hand, by viewing $H$ as a subdirect product of $\prod_{\alpha=1}^dH_{I_\alpha}$, we have that the indices $(I_\alpha)_{\alpha=1}^d$ are pairwise independent.
By Lemma~\ref{Lemma:SubdirectIndependent}, for every $1\leq \alpha\leq d$, the group $H$ contains $D_\alpha$, a finite index subgroup of $H_{I_\alpha}$.
Since $H_{I_\alpha}$ is a diagonal subgroup of $\prod_{i\in I_\alpha}G_i\cong G_{\alpha}^{\abs{I_\alpha}}$, the subgroup $D_\alpha$ is a diagonal subgroup of $L_\alpha^{\abs{I_\alpha}}$ for some finite index subgroup $L_\alpha$ of $G_{\alpha}$.
It is clear that $\prod_{\alpha=1}^dD_\alpha$ has finite index in $H$, and therefore $H$ is a virtually diagonal by block subgroup of $\prod_{i=1}^n G_i$.
\end{proof}
Finally, we 
show that the set of virtually diagonal by block subgroups is closed upon taking finite index subgroups.
\begin{lemma}\label{Lemma:SubVirtDiag}
Let $H$ be a virtually diagonal by blocks subgroup of a finite product of groups $\prod_{i=1}^nG_i$.
Then every finite index subgroup $K$ of $H$ is also a virtually diagonal by blocks subgroup.
\end{lemma}
\begin{proof}
Let $B\coloneqq\prod_{\alpha\in\Delta}D_\alpha$ be the finite index diagonal by block subgroup contained in $H$, with the notation from Definition~\ref{Def:BlockStructure}, where $D_\alpha$ is a diagonal subgroup of $L_\alpha^{\abs{I_\alpha}}$ for some finite index subgroup $L_\alpha$ of $G_\alpha$.
Then we claim that the subgroup $B'\coloneqq\prod_{\alpha\in\Delta}K\cap D_\alpha$ is a diagonal by blocks subgroup of $\prod_{i\in I}G_i$, which has finite index in $K$.
Indeed, since $K$ has finite index in $H$, the subgroup $K\cap D_\alpha$ has finite index in $D_\alpha$ for every $\alpha\in\Delta$.
In particular, $K\cap D_\alpha$ is a diagonal subgroup of $L_\alpha'^{\abs{I_\alpha}}$ for some finite index subgroup $L_\alpha'$ of $L_\alpha$. We conclude that $L_\alpha'$ has finite index in $G_\alpha$ and that $B'$ is a diagonal by blocks subgroup.
Finally, the index of $B'$ in $B$ is bounded above by the finite product $\prod_{\alpha\in \Delta}[D_\alpha:K\cap D_\alpha]$ and is hence finite.
This implies that $B'$ has finite index in $H$, and so has finite index in $K$.
\end{proof}
%
%
%
%
%
%
%
%
%
%
\section{Block structure in branch groups}\label{Section:ProofsBranch}
We begin this section with a consequence of Lemma~\ref{Lemma:subdirectProductJustInf}.
While this result is not used in the proof of the main results, we include it as another application of subdirect products techniques.
\begin{lemma}
Let $A$ be a subgroup of the first Grigorchuk group $\Grig$.
Suppose that both first level sections of $\Stab_A(\level1)$ are equal to $\Grig$.
Then $A$ contains $D\times D$ for some finite index subgroup $D$ of $\Grig$ and is thus of finite index.
\end{lemma}
\begin{proof}
Since $\Stab_A(\level1)$ has finite index in $A$, we can suppose that $A$ is a subgroup of $H\coloneqq\Stab_\Grig(\level1)$ and hence of $\Grig\times\Grig$.
By Lemma~\ref{Lemma:subdirectProductJustInf}, if both sections of $A$ are equal to $\Grig$, then either $A$ contains $D\times D$ for some finite index subgroup~$D$, or $A$ is of the form $\diag\bigl(\Grig\times\psi(\Grig)\bigr)$ for some automorphism of $\Grig$.
We argue that the second case may not happen.
More precisely, $\diag\bigl(\Grig\times\psi(\Grig)\bigr)$ is never a subgroup of $\Grig$.
Indeed, if it was the case, then $c\cdot \bigl(a,\psi(a)\bigr)=\bigl(1,d\psi(a)\bigr)$ would be in $\Grig$.
But then, $d\psi(a)$ is in $\varphi_1\bigl(\Rist_\Grig(1)\bigr)=B$ and hence $\psi(a)$ belongs to $dB\subseteq H$.
In particular, we would have that $a$ belongs to $\psi^{-1}(H)=H$ since $H$ is a characteristic subgroup of $\Grig$, which is our desired contradiction.
\end{proof}
We will use many times and without explicit reference the following easy lemma.
\begin{lemma}\label{Lemma:LevelVSTransversal}
Let $G$ be any subgroup of $\Aut(T)$ and $H$ a subgroup of $G$.
Then the following are equivalent.
\begin{enumerate}
\item \label{item:11}
There exists $n$ such that for every $v\in\level{n}$ the section $\varphi_v\bigl(\Stab_H(\level{n})\bigr)$ is either trivial or has finite index in $\varphi_v\bigl(\Stab_G(v)\bigr)$,
\item \label{item:12}There exists a transversal $X$ of $T$ such that for every $v\in X$ the section $\varphi_v\bigl(\Stab_H(X)\bigr)$ is either trivial or has finite index in $\varphi_v(G)$.
\end{enumerate}
\end{lemma}
\begin{proof}
It is clear that Property~\ref{item:11} implies Property~\ref{item:12}.
For the converse, let $H$ be a subgroup such that the conclusion of Property~\ref{item:12} holds for some $X$ and let $n$ be the maximal level of a vertex in $X$.
Then $\Stab_H(\level{n})$ is a finite index subgroup of $\Stab_H(X)$.
Now, let $w$ be a vertex of level $n$ and $v$ its unique ancestor in $X$.
Then $\varphi_w\bigl(\Stab_H(\level{n})\bigr)$ has finite index in $\varphi_w\bigl(\Stab_H(X)\bigr)=\varphi_w\bigl(\varphi_v\bigl(\Stab_H(X)\bigr)\bigr)$.
This implies that $\varphi_w\bigl(\Stab_H(\level{n})\bigr)$ is either a finite index subgroup of the trivial group (and hence trivial) or a finite index subgroup of a finite index subgroup of $\varphi_w\bigl(\Stab_G(w)\bigr)$.
\end{proof}
The subgroup induction property admits an equivalent characterization that will be more suitable for our aim.
\begin{proposition}\label{Proposition:SubgroupInductionW}
Let $G\leq \Aut(T)$ be a self-similar group.
Then the following are equivalent.
\begin{enumerate}
\item $G$ has the subgroup induction property,
\item \label{item:3}For every finitely generated subgroup $H\leq G$, there exists a transversal $X$ of $T$ such that for every $v\in X$ the section $\varphi_v\bigl(\Stab_H(X)\bigr)$ is either trivial or has finite index in $\varphi_v(G)$.
\end{enumerate}
\end{proposition}
\begin{proof}
We first show that the subgroup induction property implies the second assertion.
Let $\mathcal{X}$ be the class of subgroups $H$ of $G$ such that there exists a transversal $X$ with $\varphi_v\bigl(\Stab_H(X)\bigr)$ trivial or of finite index in $\varphi_v(G)$ for every $v$ in $X$.
We have to show that this class is inductive.

It is obvious that both $\{1\}$ and $G$ belong to $\mathcal{X}$ --- in both cases we can take $X$ to be the root of $T$.

Let $H$ and $L$ be two subgroups of $G$ such that $H$ is a finite index subgroup of $L$.
We need to show that $H$ is in $\mathcal X$ if and only if $L$ is.
For every $X$ the subgroup $\Stab_H(X)$ has finite index in $\Stab_L(X)$ and hence $\varphi_v\bigl(\Stab_H(X)\bigr)$ has finite index in $\varphi_v\bigl(\Stab_L(X)\bigr)$.
In particular, if $L$ is in $\mathcal X$, so is $H$, and this for the same transversal $X$.
On the other hand, if $H$ is in $\mathcal X$ with transversal $X$, then for every $v\in X$ the subgroup $\varphi_v\bigl(\Stab_L(X)\bigr)$ is either finite or has finite index in $\varphi_v(G)$.
Let $v$ be a vertex of $X$ such that $\varphi_v\bigl(\Stab_L(X)\bigr)$ is finite.
There exists a transversal $X_v$ of $T_v$ such that $\Stab_{\varphi_v\bigl(\Stab_L(X)\bigr)}(X_v)$ is trivial.
Let $X'$ be the transversal obtained from $X$ by removing all $v\in X$ such that $\varphi_v\bigl(\Stab_L(X)\bigr)$ is finite and replacing them by the corresponding $X_v$.
Then, for every $w\in X'$, either $\varphi_w\bigl(\Stab_L(X')\bigr)$ is trivial or it has finite index in $\varphi_w\bigl(\Stab_G(w)\bigr)$.

Finally, let $\{1,\dots, d\}$ be the first level vertices of $T$ and let $H$ be a finitely generated subgroup of $\Stab_G(\level{1})$ such that all its sections $H_1,\dots,H_d$ belong to $\mathcal X$.
Each of the $H_i$ comes with a transversal $X_i$ of $T_i$. Let $X$ be the union of these $X_i$; it is a transversal for $T$.
Let $H'\coloneqq\Stab_H(X)$. This is a finite index subgroup of $H$ with first level sections $H_i'\coloneqq \varphi_i(H')$ of finite index in $H_i$.
Therefore, for every $v$ in $X_i$, the section $\varphi_v(H_i')=\varphi_v(H')$ is either trivial or of finite index in $\varphi_v(G)$.
This implies that $H'$ belongs to $\mathcal X$ and so does $H$ since $\mathcal X$ is closed under finite extensions.

Now, suppose that $G$ satisfies Property~\ref{item:3} and let $\mathcal X$ be an inductive class of subgroups.
In particular, $\mathcal X$ contains $\{1\}$ and every finite index subgroup of~$G$.
Let $H$ be a finitely generated subgroup of $G$ and $X$ be the corresponding transversal. We need to show that $H$ is in $\mathcal X$.
Let $H'\coloneqq\Stab_H(X)$. Since $H'$ has finite index in $H$, it is finitely generated and, if it belongs to $\mathcal X$, so does $H$.
If $X$ consists only of the root of $T$, then either $H$ is finite, or it has finite index in $G$. In both cases, it belongs to $\mathcal X$.
Otherwise, let $v$ be a vertex not in $X$ but such that all its children $w_1,\dots,w_d$ are in $X$.
Since $H'$ stabilizes $X$, all the $\varphi_{w_i}\bigl(\varphi_v(H')\bigr)=\varphi_{w_i}(H')$ are in $\mathcal X$ and $\varphi_v(H')$ is a finitely generated subgroup of $G$ stabilizing the first level of $T_v$. In particular, $\varphi_v(H')$ belongs to $\mathcal X$.
Let $X_1\coloneqq v\cup X\setminus\{w_1,\dots,w_d\}$. By the above, $H'$ pointwise stabilizes $X_1$ and all its sections along $X_1$ are in $\mathcal X$.
By induction, there exists $m$ such that $X_m$ is the root and $H'$ belongs to $\mathcal X$.
\end{proof}
\begin{remark}
The alternative characterization of the subgroup induction property given by Proposition~\ref{Proposition:SubgroupInductionW} is natural even in a non-self-similar context.
From now on, we will say that a group $G$ has the \emph{subgroup induction property} if it satisfies this alternative characterization.

It is in fact possible to slightly modify the original definition of an inductive class of subgroups $\mathcal X$ (which is replaced by an inductive family $(\mathcal X_v)_{v\in T}$ of classes of subgroups) to take into account groups that are not necessarily self-similar in the Definition~\ref{Definition:SubgroupInduction}.
Proposition~\ref{Proposition:SubgroupInductionW} and its proof extend to this general context.
The details are left to the interested reader.
\end{remark}

We now derive some corollaries from this more general subgroup induction property.
The first one is about the rank of $G$.
\begin{corollary}\label{Cor:LocallyFinite}
Let $G\leq \Aut(T)$ be a branch group with the subgroup induction property.
Then $G$ is either finitely generated or locally finite.
\end{corollary}
\begin{proof}
We can suppose that $G$ is not finitely generated as otherwise there is nothing to prove.

If $H$ is a finitely generated subgroup of $G$, there exists $n$ such that for every vertex $v$ of level $n$ the section $\varphi_v\bigl(\Stab_H(\level{n})\bigr)$ is either trivial or has finite index in $\varphi_v(G)$.
Then $\Rist_G(\level{n})\cap H$ has finite index in $H$ and for every $v\in \level n$ the subgroup $\varphi_v(\Rist_G(\level{n})\cap H)$ is either trivial or has finite index in $\varphi_v(G)$.
Since $\Rist_G(\level{n})\cap H$ has finite index in $H$, it is finitely generated.

On the other hand, for any $v\in\level{n}$, the number of generators of $\Rist_G(\level{n})\cap H$ is at least the number of generators of $\varphi_v(\Rist_G(\level{n})\cap H)$.
Since $G$ is branch but not finitely generated, the rigid stabilizer $\Rist_G(\level{n})$ is not finitely generated and so is $\Rist_G(w)$ for any $w$.
Therefore, for any $w\in T$, the subgroups $\varphi_w\bigl(\Rist_G(w)\bigr)$ and $\varphi_w(G)$ are not finitely generated.
Altogether, this implies that for every $w\in \level n$, the subgroup $\varphi_w(\Rist_G(\level{n})\cap H)$ is trivial.
That is, the subgroup $\Rist_G(\level{n})\cap H$ is itself trivial and $H$ is finite.
\end{proof}
We also have a result about finite subgroups of $G$
\begin{corollary}\label{Cor:FiniteNoProjection}
Let $G\leq \Aut(T)$ be a weakly branch group with the subgroup induction property and let $H$ be a finitely generated subgroup of $G$.
Then $H$ is finite if and only if no section $\varphi_v(H)$ has finite index in $\varphi_v(G)$.
\end{corollary}
\begin{proof}
One direction is trivial.
For the other direction, let $H$ be a finitely generated subgroup of $G$.
Then there exists a transversal $X$ such that for $v\in X$ the section $\varphi_v\bigl(\Stab_H(X)\bigr)$ is either trivial or of finite index in $\varphi_v(G)$.
If no section of $H$ has finite index, then all the sections of $\Stab_H(X)$ along $X$ are trivial, that is, $\Stab_H(X)$ is trivial and $H$ is finite.
\end{proof}
This allows us to obtain a generalization of Theorem 3 of~\cite{MR3513107}.
A branch group $G\leq\Aut(T)$ is said to have \emph{trivial branch kernel},~\cite{MR2891709}, if every finite index subgroup of $G$ contains the rigid stabilizer of a level, or equivalently if the topology defined by the $\Rist_G(\level{n})$ coincide with the profinite topology.
This property is implied by the more well-known \emph{congruence subgroup property} which says that every finite index subgroup of $G$ contains the stabilizer of a level.
See~\cite{MR2891709} for more on these properties.
\begin{corollary}\label{Cor:FiniteNoProjectionPlus}
Let $G\leq \Aut(T)$ be regular branch over a subgroup $K$. Suppose that $G$ has the subgroup induction property and trivial branch kernel.
Suppose also that there exists a vertex $v$ of $T$ with $\varphi_v\bigl(\Stab_K(v)\bigr)=G$.

Let $H$ be a finitely generated subgroup of $G$.
Then $H$ is finite if and only if no section of $H$ is equal to $G$.
\end{corollary}
\begin{proof}
By Corollary~\ref{Cor:FiniteNoProjection}, if $H$ is infinite, it has a section $\varphi_v\bigl(\Stab_H(v)\bigr)$ which has finite index in $G$.
Since every section of $\varphi_v\bigl(\Stab_H(v)\bigr)$ is a section of $H$, it is sufficient to show that every finite index subgroup $N$ of $G$ has a section which is equal to $G$.

By the triviality of the branch kernel, $N$ contains the rigid stabilizer of a level and hence contains $\Rist_G(w)$ for all $w$ deep enough in the tree.
On the other hand, since $G$ is regular branch over $K$, the section $\varphi_w\bigl(\Rist_G(w)\bigr)$ contains $K$ and we are done.
\end{proof}
In the next two lemmas we prove that, under the right hypothesis, $H\leq G$ contains a block subgroup $A$ with $[H:A]<\infty$ if and only if there exists a transversal $X$ of $T$ such that for every $v\in X$ the section $\varphi_v\bigl(\Stab_H(X)\bigr)$ is either trivial or has finite index in $\varphi_v(G)$.
\begin{lemma}\label{Lemma:ProofOfMainThm}
Let $G\leq \Aut(T)$ be a just infinite branch group and let $H\leq G$ be a subgroup.
Suppose that there exists a transversal $X$ of $T$ such that for every $v\in X$ the section $\varphi_v\bigl(\Stab_H(X)\bigr)$ is either trivial or has finite index in $\varphi_v(G)$.
Then $H$ contains a block subgroup $A$ with $[H:A]<\infty$.
\end{lemma}
\begin{proof}
Let $H\leq G$ be a subgroup and suppose that there exists a transversal $X$ of $T$ such that for every $v\in X$ the section $\varphi_v\bigl(\Stab_H(X)\bigr)$ is either trivial or has finite index in $\varphi_v(G)$.
Since $G$ is just infinite and branch, the groups $\varphi_v(G)$ are also just infinite and branch \mbox{\cite[Lemma 5.1]{L2019}}, and hence not virtually abelian.

Since $G$ is branch, the group $K\coloneqq H\cap \prod_{v\in X}\Rist_G(v)$ has finite index in~$H$, and for every $v\in X$ the section $\varphi_v(K)$ is either trivial or has finite index in $\varphi_v(G)$.
A variation of Theorem~\ref{Thm:SubdirectBloc} for branch groups implies that $K$ is virtually diagonal by blocks in the sense of Definition~\ref{Def:BlockStructure}.
More precisely, there exists a partition of $U$ into $U=\bigsqcup_{i=1}^nU_i$ orthogonal, pairwise orthogonal subsets of vertices of $T$.
Moreover, the subgroup $K$ contains $K'\coloneqq\prod_{i=1}^nD_i$ as a finite index subgroup, where the $D_i$s are diagonal subgroups.
Since $K'$ has finite index in $K$ and hence in $H$, for $u\in U$ in the support of $D_i$, the section $\varphi_u(D_i)$ has finite index in $\varphi_u(H)$ and hence in $\varphi_u\bigl(\Rist_G(u)\bigr)$.
Therefore, the $D_i$s are diagonal subgroups in the sense of Definition~\ref{Def:Diagonal} and $K'$ is a block subgroup in the sense of Definition~\ref{Def:BlockStructureBranch}, of finite index in $H$.
\end{proof}
In order to prove the following lemma, we introduce a partial order on the  sets of transversals.
For two transversals $Y$ and $X$ of $T$, we say that $Y\geq X$ if every element $y$ of $Y$ is the descendant of some $x$ in $X$.
\begin{lemma}\label{Lemma:WeakSubgroupInduction}
Let $G\leq\Aut(T)$.
Let $H\leq G$ be a subgroup containing a block subgroup $A$ with $[H:A]<\infty$.
Then there exists a transversal $X$ of $T$ such that for every $v\in X$ the section $\varphi_v\bigl(\Stab_H(X)\bigr)$ is either trivial or has finite index in $\varphi_v(G)$.
\end{lemma}
\begin{proof}
Let $A=\prod_{i=1}^kD_i\leq G$ be a block subgroup and for each $i$, let $U_i$ be the support of the diagonal subgroup $D_i$. In particular, the $U_i$ are orthogonal and pairwise orthogonal.
Let $X$ be the unique minimal transversal of $T$ containing all the $U_i$s.
In particular, $\Stab_A(X)=\Stab_A(\bigcup_{i=1}^kU_i)$.
Then $\varphi_v\bigl(\Stab_A(X)\bigr)$ is either trivial or has finite index in $\varphi_v(G)$ for every $v\in X$.
If $A$ has finite index in some subgroup $H\leq G$, the subgroup $\varphi_v\bigl(\Stab_H(X)\bigr)$ is either finite or has finite index in $\varphi_v(G)$ for every $v\in X$.
Therefore, there exists another transversal $Y\geq X$ such that $\varphi_v\bigl(\Stab_H(Y)\bigr)$ is either trivial or has finite index in $\varphi_v(G)$ for every $v\in Y$.
\end{proof}
In order to prove our main result, we will need the following:
\begin{proposition}[\cite{FL2019}]\label{Prop:JustInfinite}
Let $G\leq\Aut(T)$ be a finitely generated branch group with the subgroup induction property.
Then $G$ is torsion and just infinite.
\end{proposition}
%
%
%
%
We are now able to prove a slightly  more precise version of Theorem~\ref{Thm:MainThm}.
\begin{theoremmain}
Let $G$ be a finitely generated branch group.
Then the following are equivalent.
\begin{enumerate}
\item
A subgroup $H$ of $G$ is finitely generated if and only if it contains a block subgroup $A$ with $[H:A]<\infty$,
\item
A subgroup $H$ of $G$ is finitely generated if and only if there exists $n$ such that for every $v\in \level{n}$ the section $\varphi_v\bigl(\Stab_H(\level{n})\bigr)$ is either trivial or has finite index in $\varphi_v\bigl(\Stab_G(v)\bigr)$,
\item
A subgroup $H$ of $G$ is finitely generated if and only if there exists a transversal $X$ of $T$ such that for every $v\in X$ the section $\varphi_v\bigl(\Stab_H(X)\bigr)$ is either trivial or has finite index in $\varphi_v\bigl(\Stab_G(v)\bigr)$.
\end{enumerate}
If moreover $G$ is self-similar, then this is equivalent to 
\begin{enumerate}
\setcounter{enumi}{3}
\item 
The group $G$ has the subgroup induction property in the sense of Definition~\ref{Definition:SubgroupInduction}.
\end{enumerate}
\end{theoremmain}
\begin{proof}
First of all, if $G$ is a finitely generated branch group, then every block subgroup is also finitely generated and therefore, if $H\leq G$ contains a block subgroup with $[H:A]<\infty$, then $H$ is finitely generated.
Hence, Theorem~\ref{Thm:MainThm} follows from the fact that the following are equivalent:
\begin{enumerate}\renewcommand{\theenumi}{\alph{enumi}}
\item\label{item:32}
Every finitely generated subgroup $H$ of $G$ contains a block subgroup $A$ with $[H:A]<\infty$,
\item\label{item:33bis}
For every finitely generated subgroup $H$ of $G$ there exists $n$ such that for every $v\in \level{n}$ the section $\varphi_v\bigl(\Stab_H(\level{n})\bigr)$ is either trivial or has finite index in $\varphi_v\bigl(\Stab_G(v)\bigr)$,
\item\label{item:33}
For every finitely generated subgroup $H$ of $G$ there exists a transversal $X$ of $T$ such that for every $v\in X$ the section $\varphi_v\bigl(\Stab_H(X)\bigr)$ is either trivial or has finite index in $\varphi_v\bigl(\Stab_G(v)\bigr)$,
\item\label{item:31} (If $G$ is self-similar)
The group $G$ has the subgroup induction property.
\end{enumerate}
Properties~\ref{item:33bis} and~\ref{item:33} are equivalent by Lemma~\ref{Lemma:LevelVSTransversal} and if $G$ is self-similar they are equivalent to Property~\ref{item:31} by Proposition~\ref{Proposition:SubgroupInductionW}.
Moreover, by Lemma~\ref{Lemma:WeakSubgroupInduction}, Property~\ref{item:33} is implied by Property~\ref{item:32}.
On the other hand, if $G$ has Property~\ref{item:33bis}, then it is just infinite by Proposition~\ref{Prop:JustInfinite} and hence has Property~\ref{item:32} by Lemma~\ref{Lemma:ProofOfMainThm}.
\end{proof}
We finally derive Theorem~\ref{Thm:LERF} from Theorem~\ref{Thm:MainThm}.
\begin{proof}[Proof of Theorem~\ref{Thm:LERF}]
Let $G$ be a finitely generated branch group with the congruence subgroup property and the subgroup induction property.

Let $H$ be a finitely generated subgroup of $G$.
By Theorem~\ref{Thm:MainThm}, there exists a block subgroup $A\leq H$ with $[H:A]$ finite.
By~\cite[Lemma 6.7]{L2019}, $A$ is closed in the profinite topology and so is $H$.
We have just proved that $G$ is subgroup separable, as desired.
\end{proof}

\begin{center}
\textsc{Rostistlav Grigorchuk, Mathematical Department, Texas A\&M University, College Station, TX~77843\--368,~USA}

\textit{E-mail adress: }\texttt{Grigorch@math.tamu.edu}\\[2ex]

\textsc{Paul-Henry Leemann, Institut de Math\'ematiques, Universit\'e de Neuch\^atel, Rue Emile-Argand 11, 2000~Neuch\^atel, Switzerland}\\

\textit{E-mail adress: }\texttt{Paul-Henry.Leemann@unine.ch}\\[2ex]

\textsc{Tatiana Nagnibeda, Section de math\'ematiques, Universit\'e de Gen\`eve, 7-9, rue du Conseil G\'en\'eral, 1205~Gen\`eve, Switzerland\\%
and\\%
Laboratory of Modern Algebra and Applications, St. Petersburg State University, St.Petersburg, Russia}\\
\textit{E-mail adress: }\texttt{Tatiana.Nagnibeda@unige.ch}
\end{center}
\end{document}